\documentclass{article}
\usepackage[utf8]{inputenc}
\usepackage{amsmath, amsthm, amssymb, enumerate, multicol, chngcntr}
\usepackage{xcolor}
\usepackage{pdfpages}
\usepackage{comment}

\newtheorem{prevtheorem}{Theorem}

\theoremstyle{plain}
\newtheorem{Thm}{Theorem}[section]

\newtheorem{Cor}[Thm]{Corollary}
\newtheorem{Lem}[Thm]{Lemma}

\theoremstyle{definition}
\newtheorem{Def}[Thm]{Definition}

\newcommand{\cA}{$\cal A$}
\newcommand{\cB}{$\cal B$}

\counterwithin{equation}{section}

\title{Exponent-Critical Groups}
\author{Simon R. Blackburn\footnote{Corresponding Author: Simon R. Blackburn, Royal Holloway, University of London, S.Blackburn@rhul.ac.uk}, William Cocke\footnote{William Cocke, Language Technology Institute, Carnegie Mellon University, cocke@cmu.edu}, Andrew Misseldine\footnote{Andrew Misseldine, Southern Utah University, andrewmisseldine@suu.edu}\\
and Geetha Venkataraman\footnote{Geetha Venkataraman, Dr. B. R. Ambedkar University Delhi, geetha@aud.ac.in}}

\begin{document}
\maketitle
\begin{abstract}
We define and investigate the property of being `exponent-critical' for a finite group. A finite group is said to be exponent-critical if its exponent is not the least common multiple of the exponents of its proper non-abelian subgroups. We explore properties of exponent-critical groups and give a characterization of such groups. This characterization generalises a classical result of Miller and Moreno on minimal non-abelian groups; interesting families of $p$-groups appear. 
 \end{abstract}

\section{Introduction}

Often times in group theory questions about a group can be reduced to questions about its proper subgroups. This is especially true of various first order properties of groups. For example, the question ``does a group contain an element of a certain order?" can be answered by examining all cyclic subgroups of the group. The question ``is a given finite group solvable?" can be answered by examining all $2$-generated subgroups of the group~\cite{F1995}. A more complicated and celebrated result is Thompson's classification of $N$-groups, finite groups all of whose subgroups are either solvable or Fitting-free~\cite{T1968}. A famous and classical line of research involved the question of what can be said about a non-abelian group all of whose proper subgroups are abelian. These groups are known as {\it minimal non-abelian} and were studied by Miller and Moreno in 1903 \cite{GAMHCM1903}. In this article, we introduce the following question: ``What does the exponent of the non-abelian proper subgroups of $G$ imply about the exponent of $G$?" More precisely, we make the following definition.

\begin{Def}
A finite group $G$ is \emph{exponent-critical} if the exponent of $G$ is not the least common multiple of the exponents of its proper non-abelian subgroups.
\end{Def}
To give an example, we observe that the dihedral group $D_{16}$ of order $16$ is exponent-critical: it has exponent $8$, and any proper non-abelian subgroup of $D_{16}$ is isomorphic to the group $D_8$ of exponent $4$. As a non-example, we note that $D_{24}$ is not exponent-critical: it has exponent $12$ and contains proper non-abelian subgroups isomorphic to $D_{12}$ (of exponent $6$) and $D_8$ (of exponent~$4$).

We remark that if we drop the non-abelian condition in our definition above, the problem becomes trivial. Indeed, the family of finite groups whose exponent is not the least common multiple of the exponents of its proper subgroups is exactly the family of non-trivial cyclic groups of prime power order.

In this paper, we investigate (finite) exponent-critical groups. An abelian group $G$ is exponent-critical if and only if it is non-trivial, so it is the non-abelian case which is interesting. The exponent $\exp(G)$ of a non-cyclic group $G$ is the least common multiple of the exponents of its maximal subgroups, as every cyclic subgroup is contained in a maximal subgroup. So at least one maximal subgroup of a non-abelian exponent-critical group must be abelian. But a finite group with an abelian maximal subgroup is solvable. (This result originally appeared in a paper by Herstein \cite{H1958}, is a weakening of a result in Scott \cite[13.4.6]{WRS1987}, and is a homework problem in Dixon and Mortimer \cite[3.4.7]{DM}.) So, in particular, we may make use of the theory of Hall subgroups when investigating exponent-critical groups.

Every minimal non-abelian group is exponent-critical, so our characterization of exponent-critical groups is as an extension of Miller and Moreno's characterization of minimal non-abelian groups. Since minimal non-abelian groups are well-studied, it is especially interesting to find exponent-critical groups that are not minimal non-abelian. We will construct many such examples. 
 
In this paper we show that exponent-critical groups must lie in several explicitly defined families, and all groups in these families are exponent-critical. Before discussing our results, we define the following (standard) notation. For a group $G$, we write $Z(G)$ for the center of $G$ and $\Phi(G)$ for the Frattini subgroup of $G$. If $H$ is a subgroup $G$ (written $H\leq G$), we write $\mathbf{N}_G(H)$ for the normalizer of $H$ in $G$ and $\mathbf{C}_G(H)$ for the centralizer of $H$ in $G$. We will write $\exp(G)$ for the exponent of $G$. For an element $g\in G$, we write $o(g)$ for the order of $g$.

We first show that exponent-critical groups cannot be divisible by a large number of distinct primes: 

\begin{prevtheorem}\label{prevthm:4primes}
The order of a non-abelian exponent-critical group is divisible by at most three distinct primes.
\end{prevtheorem}

The exponent-critical groups divisible by exactly three primes are classified as follows:

\begin{prevtheorem}\label{prevthm:3primes}
Let $G$ be a non-abelian finite group whose order is divisible by exactly three distinct primes. Then $G$ is exponent-critical if and only if $G$ is a direct product of a cyclic Sylow subgroup of $G$ and its complement, which is minimal non-abelian.
\end{prevtheorem}
The complement in Theorem~\ref{prevthm:3primes} is a minimal non-abelian group $H$ whose order is divisible by exactly two primes. Such groups $H$ are well understood: see Mastnak and Radjavi~\cite[Section~2.2]{MM09}, for example, for a description of Miller and Moreno's classification of such groups using modern notation.

We now turn to the case when an exponent-critical group has order divisible by exactly two primes $p$ and $q$. We need the following definition in order to state our theorem. Recall that the \emph{$p$-part} of a natural number $n=p^am$ with $\gcd(p,m)=1$ is $p^a$.
\begin{Def}
Let $G$ be a finite group, and let $p$ be a prime number. A subgroup $H$ of a group $G$ is a \emph{$p$-witness for $G$} if it is non-abelian, proper, and the $
p$-parts of $\exp(G)$ and $\exp(H)$ are equal. 
\end{Def}
It is not hard to see that a finite group $G$ is exponent-critical if and only if there exists a prime $p$ dividing the order of $G$ such that no $p$-witness for $G$ exists.
\begin{prevtheorem}\label{prevthm:2primes}
Let $G$ be a non-abelian exponent critical group whose order is divisible by exactly two distinct primes $p$ and $q$. Without loss of generality, swapping $p$ and $q$ if necessary, we may suppose that $G$ does not possess a $p$-witness for $G$.  Then $G$ is isomorphic to one of the following families of exponent-critical groups. 
\begin{enumerate}[\rm{(}i\rm{)}]
\item $G$ is a direct product of a cyclic Sylow $p$-subgroup of $G$ and its complement which is minimal non-abelian.
\item $G$ is a semi-direct product of a normal abelian Sylow $p$-subgroup $P$ of $G$ by a cyclic Sylow $q$-subgroup $Q$. The subgroup $P$ is a direct product of cyclic groups of order $p^m$ for some positive integer $m$. We have  $|G: \mathbf{C}_G(P)| =q$, and $Q$ acts non-trivially and irreducibly on $P/P^p$. 
\item $G$ is a semi-direct product of a normal abelian Sylow $p$-subgroup $P$ of $G$ by a cyclic Sylow $q$-subgroup $Q$. The subgroup $P$ is a direct product of a cyclic group of order $p^m$ and an elementary abelian group. We have $m>1$. The action of the subgroup $Q$ on $P$ preserves this direct product, acting non-trivially and irreducibly on the elementary abelian group, and trivially on the cyclic group of order $p^m$. We have $|G: \mathbf{C}_G(P)| =q$.
\item $G$ is a semi-direct product of a (unique) normal Sylow $q$-subgroup $Q$ by a non-normal cyclic Sylow $p$-subgroup $P$. The subgroup $Q$ is special in the sense of Gorenstein~\cite[p.~183]{Gorenstein}, so $Q$ is either elementary abelian or $Q$ has nilpotency class $2$ with $Q'=Z(Q)=\Phi(Q)$ and $Q'$ is elementary abelian. Further, if $x \in P$ generates $P$ then $x$ acts irreducibly on $Q/Q'$ and trivially on $Q'$.
\end{enumerate}
\end{prevtheorem}

Again, we comment that the structure of the complement in part~(i) of this theorem, namely the structure of a minimal non-abelian $q$-group, is well understood; see~\cite[Section~2.1]{MM09} for example.

It remains to consider exponent-critical $p$-groups $P$. The problem naturally splits into two cases, depending on the number of abelian maximal subgroups of $P$. We say that an exponent-critical $p$-group $P$ is of \emph{type \cA} if $P$ has exactly one maximal subgroup which is abelian, and is of \emph{type \cB} if $P$ has more than one abelian maximal subgroup. (If $a\in P$ is an element of maximal order in a non-abelian exponent-critical group $P$, then every maximal subgroup of $P$ containing $a$ is abelian. So $P$ possesses at least one non-abelian maximal subgroup. Thus every exponent-critical $p$-group has type \cA\  or \cB.) The following theorem provides a characterisation of $p$-groups of type \cB:

\begin{prevtheorem}\label{prevthm:primetypeB}
A non-abelian finite $p$-group has type \cB\ if and only if it is $2$-generated with derived subgroup of order $p$.
\end{prevtheorem}

Finally, we consider exponent-critical $p$-groups of type ${\cal A}$. We construct a `universal' group as a quotient $U/{D}$ of a certain semidirect product $U$. In order to state our result we now define this semidirect product.

\begin{Def}\mbox{}
\label{def:W_U}

\begin{itemize}
\item Let $W$ be the abelian group defined by
\[
W=\langle a_0\rangle\times \langle a_1\rangle\times\cdots \times\langle a_{p-1}\rangle
\]
where $a_0$ has order $p^m$ and $a_i$ has order $p^{m-1}$ for $1\leq i\leq p-1$. 
\item Let $\phi:W\rightarrow W$ be the automorphism (see Lemma \ref{lem:automorphism}) of $W$ such that
\[
\phi(a_i)=\begin{cases}a_{i}a_{i+1}&\text{ for }0\leq i\leq p-2\\
a_{i}\prod_{j=1}^{p-1}a_{j}^{-\binom{p}{j}}&\text{ for }i=p-1.
\end{cases}
\]
\item Let $\langle b_0\rangle$ be the cyclic group generated by an element $b_0$ of order $p^{m-1}$. Define $U$ to be the semi-direct product $U=W\rtimes \langle b_0\rangle$, where $b_0$ acts on $W$ via the automorphism $\phi$. So $b_0^{-1} w b_0=\phi(w)$ for all $w\in W$.
\item Define $D\mathrel{\unlhd} U$ to be the normal subgroup of order $p$ generated by the element $a_0^{p^{m-1}}a_{p-1}^{p^{m-2}}$.
\item Let $\cal N$ be the set of normal subgroups $N$ of $U$ such that:
\[
{D}\leq N\leq \Phi(U),\text{ }
N\cap \langle a_0\rangle=\{1\}, \text{ and }
U'\not\leq N.
\]
\end{itemize}
\end{Def}

\begin{prevtheorem}\label{prevthm:primetypeA}

\begin{enumerate}[\rm{(}i\rm{)}]
\item Let $P$ be an exponent-critical $p$-group of exponent $p^m$ and type $\cal A$. Then $P$ is isomorphic to $U/N$, where $N\in {\cal N}$.
\item Any group of the form $U/N$ with $N\in{\cal N}$ is a non-abelian exponent-critical $p$-group of exponent $p^m$.
\item There are no exponent-critical $2$-groups of exponent $2^2$ and type $\cal A$. When $p$ is odd or $m\geq 3$, the converse to~{\rm(i)} holds: Any group of the form $U/N$ with $N\in{\cal N}$ is an exponent-critical $p$-group of exponent $p^m$ and type $\cal A$.
\end{enumerate} 
\end{prevtheorem}
In particular, this theorem shows that when $p$ is odd or $m\geq 3$ there exists a unique maximal exponent-critical $p$-group $U/D$ of exponent $p^m$ and type $\cal A$. We believe these $p$-groups are particularly interesting. 

We should mention that the characterization of exponent-critical groups has applications to explicit computations involving varieties of groups~\cite{CS2020, CS2021}; this was how we originally came across this problem.

The rest of the paper proceeds as follows. Section \ref{sec proofs 4-3} contains a proof of Theorem \ref{prevthm:4primes} and Theorem \ref{prevthm:3primes}. In Section \ref{sec pq} we prove Theorem \ref{prevthm:2primes}.  Finally, in Section \ref{sec p}, we characterize exponent-critical $p$-groups and prove Theorems \ref{prevthm:primetypeB} and \ref{prevthm:primetypeA}. 

\section{Exponent-critical groups of order divisible by three primes.}\label{sec proofs 4-3}

We will prove Theorem $A$ and $B$ in this section.

\begin{proof}[\bf{Proof of Theorem \ref{prevthm:4primes}}.]
Let $G$ be a non-abelian exponent-critical finite group. We mentioned in the introduction that $G$ is solvable. Suppose that the order of $G$ is divisible by four or more primes.
Since $G$ is non-abelian, it must be the case that some Sylow $p$-subgroup $P$ of $G$ is not central. Let $q$ be a prime dividing the order of $G$ such that $\mathbf{C}_G(P)$ does not contain a Sylow $q$-subgroup of $G$. Let $H_{pq}$ be a $pq$-Hall subgroup of $G$, which is necessarily non-abelian. For any other prime $r$ dividing $|G|$, a Hall $pqr$-subgroup $H_{pqr}$ of $G$ is therefore non-abelian. Moreover, $H_{pqr}$ is proper since $|G|$ is divisible by four or more primes. So $H_{pqr}$ is a $p$-witness, a $q$-witness and an $r$-witness for $G$. So there exists an $\ell$-witness for $G$ for all primes $\ell$ dividing the order of $G$, which contradicts $G$ being exponent-critical.
\end{proof}

\begin{Thm}\label{thm:3primes} Let $G$ be a non-abelian, exponent-critical group. If the order of $G$ is divisible by three distinct primes, then $G$ has a non-trivial central cyclic Sylow subgroup and all Sylow subgroups of $G$ are abelian. 
\end{Thm}

\begin{proof}
By the same reasoning as in the proof of Theorem \ref{prevthm:4primes}, the group $G$ has a non-abelian $pq$-Hall subgroup for two primes $p$ and $q$ dividing $|G|$. This subgroup is proper as three primes divide the order of $G$, and so it is a $p$-witness and $q$-witness for $G$. Suppose $r$ is the third prime dividing $|G|$. If either the $pr$- or $qr$-Hall subgroups are non-abelian, then we have an $r$-witness for $G$ and so $G$ is not exponent-critical. Therefore both these Hall subgroups are abelian. Hence all Sylow subgroups of $G$ are abelian, and the Sylow $r$-subgroup centralizes a Sylow $p$- and Sylow $q$-subgroup of $G$. If the Sylow $r$-subgroup is not cyclic, then $G$ would contain a proper subgroup which is an $r$-witness for $G$, which would imply that $G$ is not exponent-critical.
\end{proof}

We now prove Theorem \ref{prevthm:3primes}, which classifies exponent-critical groups whose orders are divisible by only $3$ distinct primes. 

\begin{proof}[\bf{Proof of Theorem \ref{prevthm:3primes}}]
Suppose $G$ is exponent-critical. Then $G$ has a central cyclic Sylow $p$-subgroup $P$ by Theorem~\ref{thm:3primes}. Let $H$ be a Hall $qr$-subgroup of $G$. Since $P$ is central, $G$ is a direct product of $P$ and $H$. If $H$ were abelian, then $G$ would be abelian. Hence $H$ is non-abelian. Suppose by way of contradiction that $H$ has a non-abelian proper subgroup $K$. Then $PK\le G$ and $\exp(G)$ is the least common multiple of $\exp(PK)$ and $\exp(H)$, which implies that $G$ is not exponent-critical. Hence $H$ is minimal non-abelian.

Conversely, suppose that $G$ is a direct product of a cyclic $p$-group $P$ and a minimal non-abelian $qr$-group $H$. Any proper subgroup of $G$ containing $P$ is abelian, since $H$ is minimal non-abelian. So no subgroup of $G$ is a $p$-witness for $G$, and hence $G$ is exponent-critical.
\end{proof} 

\section{Exponent-critical groups of order divisible by exactly two primes} \label{sec pq} 

Before proving Theorem~\ref{prevthm:2primes},  we show that each of the four families of groups given in Theorem \ref{prevthm:2primes} consists of exponent-critical groups, by showing that there is no $p$-witness for any group $G$ in the family. 

\begin{Lem}\label{lem:2primesconv_a}
Let  $p$ and $q$ be distinct primes. Let $G$ be a direct product of a normal cyclic Sylow $p$-subgroup $P$ and a minimal non-abelian $q$-subgroup $Q$. Then $G$ is exponent-critical.
\end{Lem}
\begin{proof} Suppose, for a contradiction, $H$ is a $p$-witness for $G$. Since $P$ is cyclic,  $P\leq H$. Since $H$ is proper, $H=P\times K$ for some proper subgroup $K$ of $Q$. Since $Q$ is minimal non-abelian, $K$ is abelian. But then $H$ is abelian, so $H$ is not a $p$-witness for $G$. This contradiction shows that $G$ is exponent-critical.
\end{proof}

\begin{Lem}\label{lem:2primesconv_ba}
Let $p$ and $q$ be distinct primes. Suppose $G$ is a non-abelian semi-direct product of a normal abelian Sylow $p$-subgroup $P$ of $G$ by a cyclic $q$-complement $Q$. Suppose that $P$ is the direct product of cyclic groups of order $p^m$, where $m$ is a positive integer. Furthermore, suppose that $|G:\mathbf{C}_G(P)|=q$ and $Q$ acts non-trivially and irreducibly on $P/P^p$. Then $G$ is exponent-critical.
\end{Lem}
\begin{proof}
Suppose, for a contradiction, that $H$ is a $p$-witness for $G$. We see that $H=P_1Q_1$ where $P_1=P\cap H$ and where $Q_1$ is a $q$-group. By replacing $H$ by a conjugate if necessary, we may assume that $Q_1\leq Q$. Since $H$ is non-abelian, $Q$ is cyclic, and $|G:\mathbf{C}_G(P)|=q$, we see that $Q_1=Q$. Since $P$ is normal and $Q\leq H$, we see that $P_1$ is normalised by $Q$. Because $H$ has exponent divisible by the $p$-part $p^m$ of the exponent of $G$, the quotient $P_1P^p/P^p$ is non-trivial. The action of $Q$ on $P/P^p$ is irreducible and $P_1P^p/P^p$  is $Q$-invariant, and hence $P_1P^p/P^p=P/P^p$. So $P$ is generated by $P_1$ and $P^p$, and hence $P_1$ generates $P$ (as $P^p=\Phi(P)$). Hence $P_1=P$. Thus $H=P_1Q_1=PQ=G$, and so $H$ is not proper. This contradiction establishes the lemma.
\end{proof}

\begin{Lem}\label{lem:2primesconv_bb}
Let $p$ and $q$ be distinct primes. Moreover, suppose $G$ be a non-abelian semi-direct product of a normal abelian Sylow $p$-subgroup $P$ of $G$ by a cyclic $q$-complement $Q$. For $m>1$, suppose that $P$ is the direct product of a cyclic group of order $p^m$, and an elementary abelian $p$-group. Suppose $Q$ centralises the cyclic group of order $p^m$, and acts irreducibly on the elementary abelian group. Furthermore, suppose that $|G:\mathbf{C}_G(P)|=q$. Then $G$ is exponent-critical. 
\end{Lem}
\begin{proof}
We may write $P=C\times A$ where $C$ is a cyclic group of order $p^m$ and where $A$ is elementary abelian; $Q$ centralises $C$ and acts irreducibly by conjugation on $A$. We choose a generator $y$ for $Q$, so $Q=\langle y\rangle$. 

Suppose, for a contradiction, that $H$ is a $p$-witness for $G$. As in Lemma~\ref{lem:2primesconv_ba}, we may assume that $H=P_1Q$ where $P_1$ is normalised by $Q$.

Since $H$ is non-abelian, there exists an element $g:=ca\in P_1$ where $c\in C$ and $a\in A\setminus\{1\}$. Now $Q$ acts non-trivially on $A$, since $G$ is non-abelian. Since $Q$ acts irreducibly and non-trivially on $A$, we deduce that $\mathbf{C}_A(y)=\{1\}$ and so $[g,y] \in A\setminus \{1\}$. Since $P_1$ is normalised by $Q$, we see that $[g,y]\in P_1$ and so
$P_1\cap A\not=\{1\}$. But $Q$ acts irreducibly on $A$ and so $A\leq P_1$. Since $p^m$ divides the exponent of $H$, and since $m>1$, we see that there exists an element $x\in P_1$ such that $x\notin AP^p$. Since $AP^p$ has index $p$ in $P$, we find that $\langle x\rangle AP^p=P$ and so $\langle x\rangle A=P$ (since $P^p=\Phi(P)$). But  $\langle x\rangle A\leq P_1$ and so $P_1=P$.  Thus $H=P_1Q_1=PQ=G$, and so $H$ is not proper. This contradiction establishes the lemma.
\end{proof}

\begin{Lem}\label{lem:2primesconv_c}
Let $p$ and $q$ be distinct primes. Let $G=Q\rtimes P$ be a (non-abelian) semi-direct product of a Sylow $q$-subgroup $Q$ by a non-normal cyclic Sylow $p$-subgroup $P$. Suppose that $Q$ is special, so either $Q$ is elementary abelian or $Q'=Z(Q)=\Phi(Q)$. Moreover, suppose that $P$ acts (by conjugation) trivially on $Q'$ and irreducibly on $Q/Q'$, Then $G$ is exponent-critical.
\end{Lem}

\begin{proof} Let $G = Q \rtimes P$ satisfy the conditions of the lemma. We show that there is no $p$-witness for $\exp (G)$.

Assume for a contradiction that $H$ is a $p$-witness for $\exp (G)$. Without loss of generality, we may assume that $H = Q_1P_1$ for some $Q_1 \leq Q$ and $P_1 \leq P$ where $Q_1$ is $P_1$-invariant. Clearly $P_1$ contains an element $x$ which is of maximal order in $P$, and so (since $P$ is cyclic) we see that $P_1=P=\langle x\rangle$.

Suppose that $Q_1\subseteq \Phi(Q)$. Since $Q$ is special (whether elementary abelian or not) $Q'=\Phi(Q)\leq Z(Q)$. So $Q_1\leq Z(Q)$ is abelian and, since $P$ centralises $Q'$, we see that $P$ centralises $Q_1$. But then $H$ is abelian, and we have a contradiction. We may deduce that $Q_1$ contains an element in $Q\setminus \Phi(Q)$, and so $Q_1\Phi(Q)/\Phi(Q)$ is non-trivial.

Since $Q$ is special, $\Phi(Q)= Q'$ and so $P$ acts irreducibly on $Q/\Phi(Q)$. Since $Q_1$ is $P$-invariant, the non-trivial subgroup $Q_1\Phi(Q)/\Phi(Q)$ is a $P$-invariant subgroup of $Q/\Phi(Q)$. So $Q_1\Phi(Q)/\Phi(Q)=Q/\Phi(Q)$ and thus $Q_1\Phi(Q)=Q$. By the non-generation property of the Frattini subgroup, we may deduce that $Q_1=Q$ and so $H=G$. This contradicts the fact that $H$ is proper, as required.
\end{proof}

We are now in a position to prove Theorem \ref{prevthm:2primes}. 

\begin{proof}[\bf{Proof of Theorem \ref{prevthm:2primes}}.]\label{thm:2primes} Lemmas~\ref{lem:2primesconv_a} to~\ref{lem:2primesconv_c} show that the four families described in the theorem consist entirely of exponent-critical groups. It suffices to show that any non-abelian exponent-critical $p,q$-group lies in one of these four families.

Let $G$ be a non-abelian exponent-critical group with $|G| = p^{\alpha}q^{\beta}$ where $p$ and $q$ are distinct primes. Suppose that there is no $p$-witness for $G$. So all proper subgroups containing a Sylow $p$-subgroup $P$ of $G$ are abelian, and in particular $P$ is abelian. Moreover, either $P$ is normal or $\mathbf{N}_G(P)$ is abelian. 

\vspace{.05in}
\noindent
{\bf Part I: $P$ is normal.}

\noindent Suppose that a complement $Q$ of $P$ is also normal, and so $G = P \times Q$. Since no non-abelian proper subgroup of $G$ can contain $P$ we see that $Q$ is minimal non-abelian. Now suppose that $P$ is not cyclic. Let $x$ be an element of maximal order in $P$. Then $\langle x\rangle Q$ is a $p$-witness for $G$, which is a contradiction. So $P$ must be cyclic and $G$ lies in the family described in Lemma~\ref{lem:2primesconv_a}. 

\vspace{.05in}
\noindent
Now let us consider the case when $P$ is normal but its complement, say $Q$, is not. So we can find an element $y \in Q$ which is not in $\mathbf{C}_G(P)$. Thus $P\langle y\rangle$ is a non-abelian subgroup of $G$ containing $P$. So $G = P\langle y\rangle$ and therefore $Q = \langle y\rangle$. Furthermore $\langle y^q \rangle \leq \mathbf{C}_G(P)$ since otherwise $P\langle y^q \rangle $ is a $p$-witness for $G$. So $|G: \mathbf{C}_G(P)| =q$.

Let the exponent of $P$ be $p^m$. Write $P[p^{m-1}]$ for the subgroup of elements of $P$ of order dividing $p^{m-1}$, and note that $P^p\leq P[p^{m-1}]$. Set $V=P/P^p$. We regard $V$ as a vector space over $\mathbb{F}_p$. Indeed, $V$ can be thought of as an $\mathbb{F}_pQ$-module, with the action of $Q$ derived from conjugation. Now $Q$ acts non-trivially on $P$ by conjugation, and the only automorphisms of $P$ that induce the identity on $P/\Phi(P)$ have $p$-power order. Since $\Phi(P)=P^p$, we see that $Q$ acts non-trivially on $V$.

Let $\pi:P\rightarrow V$ be the natural homomorphism. Now $U:=\pi(P[p^{m-1}])=P[p^{m-1}]/P^p$ is an $\mathbb{F}_pQ$ submodule of $V$. Indeed $U$ is a proper submodule, since $P$ has exponent $p^m$. Since the order of $Q$ is coprime to $p$, the module $V$ is completely decomposable and so we may write $V$ as the sum
\begin{equation}
\label{eqn:decomposition}
V=U_1\oplus U_2\oplus\cdots U_k\oplus W_1\oplus W_2\oplus\cdots W_\ell
\end{equation}
of irreducible submodules, where $U=U_1\oplus U_2\oplus\cdots \oplus U_k$ and where $W:=W_1\oplus W_2\oplus\cdots\oplus W_\ell$ forms a complement to $U$ in $V$. We have $k\geq 0$ and, since $U$ is proper, $\ell\geq 1$. Since all the elements in $P\setminus P[p^{m-1}]$ have order $p^m$, we see that $\pi^{-1}(W_i)$ has a subgroup of $P$ of exponent $p^m$. We divide our argument into two sub-cases:

\paragraph{Sub-case 1: Suppose that $Q$ acts non-trivially on a submodule $W_i$.} We see that $\langle\pi^{-1}(W_i), Q\rangle$ is a non-abelian subgroup of $G$ of exponent $p^m$, and (since $G$ has no $p$-witness) we deduce that $G=\langle\pi^{-1}(W_i), Q\rangle$. Hence $k=0$, $\ell=1$ and $V=W_i$ in this case. Since $k=0$ we see that $P$ is a direct product of cyclic groups of order $p^m$. Since $W_i$ is irreducible and $W_i=P/P^p$, we see that $Q$ acts irreducibly on $P/P^p$. Thus $G$ lies in the family described in Lemma~\ref{lem:2primesconv_ba}. 

\paragraph{Sub-case 2: $Q$ acts trivially on all submodules $W_i$.} Since $Q$ acts non-trivially on $V$, we see that $Q$ acts non-trivially on one of modules $U_i$. (In particular, this implies that $Q$ acts non-trivially on $U$ and so $m>1$.) The subgroup $\langle\pi^{-1}(U_i\oplus W_1), Q\rangle$ is non-abelian and of exponent $p^m$, and so $\langle\pi^{-1}(U_i\oplus W_1), Q\rangle=G$. Hence $k=1=\ell$, and $V=U_1\oplus W_1$. If $Q$ acts non-trivially by conjugation on $P^p$, the subgroup $\langle \pi^{-1}(W_1),Q\rangle$ would be a $p$-witness, and so we deduce that $Q$ centralises $P^p$.

Taking $p$-powers in $P$ induces a surjective $\mathbb{F}_pQ$-module homomorphism $f$ from $V=P/P^p$ to the $\mathbb{F}_qQ$-module $P^p/P^{p^2}$. Since $m>1$, this homomorphism is non-trivial. Since $Q$ centralises $P^p/P^{p^2}$ and acts non-trivially and irreducibly on $U_1$, we see that $f(U_1)=0$ and $f(W_1)=P^p/P^{p^2}$. Since $W_1$ is trivial and irreducible, it has dimension $1$ and so $P$ is a direct product of a cyclic group of order $p^m$ with an elementary abelian group.

Let $x\in \pi^{-1}(W_1\setminus\{0\})$. The element $x\in P$ has order $p^m$ and (since $|Q|$ is coprime to $p$) $x$ is centralised by $Q$. 
The action of $Q$ by conjugation on $P[p]$ gives $P[p]$ the structure of an $\mathbb{F}_pQ$-module. Let $A$ be a complement to the submodule in $P[p]$ generated by $x^{p^{m-1}}$. Then $P=\langle x\rangle\times A$, where the action of $Q$ fixes $x$ and preserves the direct product. Since $A P^p/P^p=U_1$, we see that the action of $Q$ on $A$ is irreducible. So $G$ lies in the family described in Lemma~\ref{lem:2primesconv_bb}. 

\vspace{.05in}
\noindent
{\bf Part II: $P$ is not normal and so $\mathbf{N}_G(P)$ is abelian}

\noindent Clearly in this case, $P\leq Z(\mathbf{N}_G(P))$ and by Burnside's Normal $p$-complement Theorem, $G$ has a normal $p$-complement $Q$. Therefore $G$ is a semi-direct product of its Sylow $q$-subgroup $Q$ by a non-normal abelian Sylow $p$-subgroup $P$.

Suppose, for a contradiction, that $P$ is not cyclic. Let $x \in P$ have maximal order. Then $\langle x\rangle Q$ is a proper subgroup, so must be abelian (otherwise we have a witness for the $p$-part of $\exp G$). So all elements of maximal order lie in $\mathbf{C}_G(Q)$. But $P$ is generated by its elements of maximal order, and so $P\leq \mathbf{C}_G(Q)$. This implies that $P$ is normal, and we have our contradiction. Hence $P$ is cyclic.

The subgroup $P$ must centralise any proper $P$-invariant subgroup of $Q$, otherwise we have a witness to the $p$-part of $\exp (G)$. So a generator $x$ of $P$ gives rise (via conjugation) to an automorphism of $Q$ that acts trivially on any proper subgroup of $Q$. This automorphism is non-trivial, since $P$ is not normal. Hence, by~\cite[Theorem 5.3.7]{Gorenstein}, $Q$ is special, $P$ acts trivially on $Q'$, and $P$ acts irreducibly on $Q/Q'$. So $G$ lies in the family described in Lemma~\ref{lem:2primesconv_c}, and the theorem follows.
\end{proof}

\section{Exponent-critical $p$-groups}\label{sec p}

We now examine exponent-critical $p$-groups. We will prove Theorems \ref{prevthm:primetypeB} and \ref{prevthm:primetypeA} in this section.

Let $P$ be a non-abelian $p$-group of order $p^n$. We denote by ${\cal M}_P$ the set of elements of $P$ of maximal order, or equivalently the set of elements that have order equal to $\exp (P)$. Note that $P$ is exponent-critical if and only if ${\cal M}_P  \cap M = \emptyset$ for all non-abelian maximal subgroups $M$ of $P$.

\begin{Thm}
\label{thm:basic_pgroup}
Let $P$ be a finite non-abelian exponent critical $p$-group. Then $P=\langle a,b\rangle$ where $a\in P$ has maximal order and $b\in P$. Moreover, $P$ is solvable of derived length $2$.
\end{Thm}
\begin{proof}
The Frattini subgroup $\Phi(P)$ is contained in all maximal subgroups. Since at least one maximal subgroup is abelian, $\Phi(P)$ is abelian. Since $P'\leq P'P^p=\Phi(P)$, and since $P$ is non-abelian, $P$ is solvable of derived length $2$.

We now show that $P$ is $2$-generated. Suppose, for a contradiction, that all elements in ${\cal M}_P$ are central. Let $a\in{\cal M}_P$, and let $x,y\in P$ be such that $[x,y]\not=1$. As $x$ is not central, it does not have maximal order. Hence, using the fact that $[a,x]=1$, we see that $xa$ has maximal order. Since $a$ is central, $[xa,y]=[x,y]^a[a,y]=[x,y]\not=1$. Thus we have a contradiction as required, and we may deduce that there exists an element $a\in{\cal M}_P$ of maximal order that is not central. Let $b \in P$ be such that $[a , b] \not = 1$. Thus $\langle a, b \rangle$ is a non-abelian subgroup of exponent $\exp (P)$ and so $P = \langle a, b \rangle$. Hence the theorem follows.
\end{proof}

\begin{proof}[\bf{Proof of Theorem \ref{prevthm:primetypeB}}.]\label{thm:primetypeB}

Suppose $P$ is a non-abelian finite group which is $2$-generated and $|P'| =p$.  We show that $P$ is exponent critical of type \cB. Let $P=\langle a,b\rangle$. Since $P$ is nilpotent, $[P', P]$ is a proper subgroup of $P'$. Thus $P$ has nilpotency class 2. In particular, $P'\leq Z(P)$. Since $P$ has nilpotency class 2 we get that $1=[a,b]^p=[a^p,b]= [a,b^p]$. So $a^p$ and $b^p$ are central. Hence $\Phi(P)\leq Z(P)$ as in a $p$-group $\Phi(P)$ is the subgroup generated by $P'$ and the $p$-th powers of a generating set. Let $M$ be a maximal subgroup. Then $\Phi(P)\leq M$. Since $P$ is $2$-generated, $\Phi(P)$ has index $p^2$ in $P$, and so has index $p$ in $M$. Since $\Phi(P)$ is central in $P$, we see that $Z(M)\geq \Phi(P)$ and so the centre of $M$ has index at most $p$ in $M$. But then $M$ must be abelian (and $Z(M)=M$). Thus $P$ is a minimal non-abelian group and hence exponent-critical of type \cB, as required.

Conversely, suppose that $P$ is exponent-critical of type~\cB. By Theorem~\ref{thm:basic_pgroup}, $P=\langle a,b\rangle$ for elements $a,b\in P$ with $a\in{\cal M}_P$. Since $P$ has type \cB, it contains distinct abelian maximal subgroups  $M_1$ and $M_2$. For $x,y\in P$ we see that $[x,y]\in P'\leq \Phi(P)\leq M_1\cap M_2$ and so $[x,y]$ commutes with all elements of $M_1$ and $M_2$. Since $P=M_1M_2$, we deduce that $P'$ is central in $P$ and so $P$ has nilpotency class 2. Hence $[x, yz] = [x, y][x, z]$ and $[xy, z] = [x, z][y, z]$ for all $x, y, z$ in $P$. Using this we see that $P'= \langle [a, b] \rangle$. Let $M_a$ be a maximal subgroup of $P$ such that $a \in M_a$. Since $a$ has maximal order and $P$ is exponent-critical we see that $M_a$ is abelian. Clearly $b^p \in M_a$. Since $P$ has nilpotency class $2$, we see that $1 =[a, b^p] = [a, b]^p = [a^p, b]$. So $o([a, b]) = p = |P'|$, and part~(iii) of the theorem follows.
\end{proof}

We comment that the isomorphism classes of $p$-groups of type \cB\ are known precisely, as part of the (more general) classifications of Blackburn~\cite{B99} or of Ahmad, Magidin and Morse~\cite{AMM2012} (see also~\cite{BK1993,KVS1999}).  Using the approach in~\cite{AMM2012}, we have the following list:

\begin{Thm}
\label{expcriticalpgrp_nc2}Let $P$ be an exponent-critical finite non-abelian $p$-group of type ${\cal B}$ of order $p^n$. Then $$ P \cong \left\langle a, b \mid {[a , b]}^{p}  =  [a , b , a] = [ a , b , b] = 1; a^{p^\alpha} = {[a , b]}^{p^\rho}\!\!, \,\,b^{p^\beta} = {[a , b]}^{p^\sigma} \right\rangle$$ where $\alpha, \beta, \rho, \sigma$ are integers such that: $\alpha \geq \beta \geq 1$, $\alpha + \beta =n-1$, and $0 \leq \rho, \sigma \leq 1$.  When $p$ is odd, $P$ is isomorphic to exactly one of the groups whose parameters $(\alpha,\beta,\rho,\delta)$ are listed below: 
\begin{enumerate}[\rm{(}1\rm{)}]
\item 
\begin{enumerate}[\rm{(}a\rm{)}]
\item $(\alpha, \beta,  0, 1)$ with $\alpha>\beta\geq 1$.
\item $(\alpha, \beta,  1, 1)$ with $\alpha>\beta\geq 1$.
\item $(\alpha, \beta,  1, 0)$ with $\alpha>\beta\geq 1$.
\end{enumerate}
\item 
\begin{enumerate}[\rm{(}a\rm{)}]
\item $(\alpha, \alpha,  0, 1)$ with  $\alpha\geq 1$.
\item $(\alpha, \alpha,  1, 1)$ with $\alpha\geq 1$.
\end{enumerate}
\end{enumerate}
When $p=2$, $P$ is isomorphic to exactly one of the groups whose parameters $(\alpha,\beta,\rho,\delta)$ are 
\begin{enumerate}[\rm{(}1\rm{)}]
\item 
\begin{enumerate}[\rm{(}a\rm{)}]
\item $(\alpha, \beta,  0, 1)$ with $\alpha>\beta\geq 1$.
\item $(\alpha, \beta,  1, 1)$ with $\alpha>\beta\geq 1$.
\item $(\alpha, \beta,  1, 0)$ with $\alpha>\beta\geq 1$.
\end{enumerate}
\item 
\begin{enumerate}[\rm{(}a\rm{)}]
\item $(\alpha, \alpha,  0, 1)$ with  $\alpha> 1$.
\item $(\alpha, \alpha,  1, 1)$ with $\alpha> 1$.
\end{enumerate}
\item
\begin{enumerate}[\rm{(}a\rm{)}]
\item $(1, 1,  0, 0)$.
\item $(1, 1,  1, 1)$.
\end{enumerate}
\end{enumerate}
\end{Thm}

We now turn to exponent-critical groups of type~\cA, with an aim of proving Theorem \ref{prevthm:primetypeA}. In the lemma below, we write $[x,_i b]$ for the $i$-times iterated commutator of $x$ and $b$. So, for example, $[x,_1 b]=[x,b]$, $[x,_3 b]=[[[x,b],b],b]$ and $[x,_0 b]=x$. We write $Z_{p^m}$ for the cyclic group of order $p^m$.

\begin{Lem}
\label{lem:Type_B_basic_facts}
Let $P$ be a non-abelian exponent-critical $p$-group of type ${\cal A}$. Let $A$ be a maximal subgroup containing an element of maximal order $p^m$.
\begin{enumerate}[\rm{(}i\rm{)}]
\item $A$ is abelian, normal, and contains all elements of order $p^m$.
\item $A\cong Z_{p^m}\times S$, where $S$ has exponent dividing $p^{m-1}$.
\item $P'$ is abelian, of exponent dividing $p^{m-1}$.
\end{enumerate}
\end{Lem}

\begin{proof}
$A$ is normal, since all maximal subgroups of a $p$-group are normal. If $A$ were non-abelian, $P$ would have a proper non-abelian subgroup of exponent $p^m$ and hence $P$ would not exponent-critical. This contradiction shows that $A$ is abelian. Since $P$ is of type ${\cal A}$, we have that $A$ is the unique maximal subgroup of $P$ which is abelian. Clearly, $A$ will contain all the elements of maximal order and part~(i) of the lemma follows.

To prove part~(ii), first note that $A$ has exponent $p^m$, and so $A\cong (Z_{p^m})^r\times S$ where $r$ is a positive integer and $S$ has exponent dividing $p^{m-1}$. Define $K$ to be the subgroup of all elements of $A$ of order dividing $p^{m-1}$. Every element of $A\setminus K$ has order $p^m$, and $A/K$ is elementary abelian of order $p^r$. Let $b\in P\setminus A$. Since $b$ acts nilpotently on $A/K$, there exists a subgroup $L$ containing $K$ such that $b\in \mathbf{N}_P(L)$ and $L$ has index $p$ in $A$. Note that $L$ is normal in $P$, since $\{b\}\cup A\subseteq \mathbf{N}_P(L)$. Consider the subgroup $M$ generated by $L$ and $b$. Since $A$ has index $p$ in $P$, we get that $b^p\in A$. Since $b\in P\setminus A$, the order of $b$ divides $p^{m-1}$ and so $b^p\in K\leq L$. Hence $M$ has index $p$ in $P$ and so is maximal. If $r>1$ then $L$ contains elements of $A\setminus K$, and so $L$ contains an element of maximal order. Hence $M$ must be abelian (as $P$ is exponent critical) and therefore $M=A$. But $b\in M$ and $b\notin A$. This contradiction shows that $r=1$, and so part~(ii) is established.

To prove~(iii), we note that $K$ has index $p$ in $A$ by part~(ii) and hence $K$ has index $p^2$ in $P$. So $P/K$ is abelian and therefore $P'\leq K$. Hence the exponent of $P'$ divides $p^{m-1}$, as required.
\end{proof}

\begin{Lem}
\label{lem:abelian_index_p_facts}
Let $P$ be a $p$-group. Suppose $P$ possesses an abelian maximal subgroup $A$. Let $b\in P\setminus A$.
\begin{enumerate}[\rm{(}i\rm{)}]   
\item The map $\kappa:A\rightarrow A$ defined by $\kappa(x)=[x,b]$ is a homomorphism.
\item For $x\in A$, and $b\in P\setminus A$,
\begin{equation}
\label{eqn:power}
(bx)^i=b^ix^i\prod_{j=1}^{i-1} [x,_j b]^{\binom{i}{j+1}}.
\end{equation}
\item For $x\in A$, and $b\in P\setminus A$,
\begin{equation}
\label{eqn:b_order}
[x,_p b]=\prod_{i=1}^{p-1}[x,_i b]^{-\binom{p}{i}}.
\end{equation}
\end{enumerate}
\end{Lem}
\begin{proof}
To prove~(i), we note that for all $x\in A$, $[x,b]\in A$ as $A$ is normal in $P$. So $[x,b]^y=[x,b]$ for all $y\in A$, since $A$ is abelian. Hence
\[
[xy,b]=[x,b]^y[y,b]=[x,b][y,b],
\]
and $\kappa$ is a homomorphism, as required.

We prove the equality~\eqref{eqn:power} by induction on $i$. The equality~\eqref{eqn:power} clearly holds when $i=1$. Assume the equality holds for a fixed value of $i$. Since $P/A$ is abelian, all commutators lie in the abelian subgroup $A$. Hence
\begin{align*}
(bx)^ibx&=b^ix^i\left(\prod_{j=1}^{i-1} [x,_j b]^{\binom{i}{j+1}}\right) bx\text{, by our inductive hypothesis}\\
&=b^{i+1}x^i[x,b]^i\left(\prod_{j=1}^{i-1} [x,_j b]^{\binom{i}{j+1}}\right) \left(\prod_{j=1}^{i-1} [x,_{j+1} b]^{\binom{i}{j+1}}\right)x\\
&\quad\text{ by part~(i)}\\
&=b^{i+1}x^{i+1}[x,b]^{\binom{i}{1}+\binom{i}{2}}\left(\prod_{j=2}^{i-1} [x,_j b]^{\binom{i}{j+1}} \right)\left(\prod_{j=2}^{i} [x,_{j} b]^{\binom{i}{j}}\right)\\
&\quad\text{ as }x\in A,\\
&=b^{i+1}x^{i+1}\prod_{j=1}^{i} [x,_j b]^{\binom{i}{j+1}+\binom{i}{j}},\text{ since }\binom{i}{i+1}=0,\\
&=b^{i+1}x^{i+1}\prod_{j=1}^{i} [x,_j b]^{\binom{i+1}{j+1}},
\end{align*}
and so part~(ii) follows by induction.

Finally, we prove the part~(iii) of the lemma. We first prove that, for any positive integer $i$,
\begin{equation}
\label{eqn:conj}
b^{-i} xb^i=x\prod_{j=1}^i [x,_j b]^{\binom{i}{j}}.
\end{equation}
To show this, we note that $b^{-1}xb=x[x,b]$ and so~\eqref{eqn:conj} holds when $i=1$. If~\eqref{eqn:conj} holds for a given value of $i$, we see that
\begin{align*}
b^{-(i+1)}xb^{i+1}&=b^{-1}x\left(\prod_{j=1}^i [x,_j b]^{\binom{i}{j}}\right)b\\
&=x[x,b]\left(\prod_{j=1}^i [x,_j b]^{\binom{i}{j}}\right)\left(\prod_{j=1}^i [x,_{j+1} b]^{\binom{i}{j}}\right)\\
&=x\prod_{j=1}^{i+1} [x,_j b]^{\binom{i}{j}+\binom{i}{j-1}}\\
&=x\prod_{j=1}^{i+1} [x,_j b]^{\binom{i+1}{j}},
\end{align*}
and so, by induction, the equation~\eqref{eqn:conj} holds for all $i$. Now $b^p\in A$ since $A$ is maximal. As $A$ is abelian, $b^{-p}xb^p=x$, and so
\[
x=x\prod_{j=1}^p [x,_j b]^{\binom{p}{j}}.
\]
Hence part~(iii) of the lemma follows.
\end{proof}

Recall that we defined a group $U$ and normal subgroup $D$ in Definition~\ref{def:W_U}. Our aim now is to construct a `universal' group for the exponent-critical $p$-groups of exponent $p^m$ and type ${\cal A}$, as the quotient $U/{D}$ of the semidirect product $U$. Before this, in the two lemmas below, we examine some of the properties of the group $U$.

\begin{Lem}
\label{lem:automorphism}
The automorphism $\phi$ defined in Definition \ref{def:W_U} has order $p$. In particular, the definition of $U$ as a semidirect product is well-defined.
\end{Lem}
\begin{proof}
We see that for $0\leq j\leq p-1$
\[
\phi^i(a_0)=\prod_{j=0}^i a_j^{\binom{i}{j}},
\]
and so a short calculation shows that $\phi^p(a_0)=\phi(\phi^{p-1}(a_0))=a_0$.
The lemma now follows since $W$ is generated by $\phi^j(a_0)$ for $0\leq j\leq p-1$.
\end{proof}

\begin{Lem}
\label{lem:basic_U_facts}
Let $U=W\rtimes \langle b_0\rangle$ be the semidirect product defined above. 
\begin{enumerate}[\rm{(}i\rm{)}]
\item We have $a_i=[a_0,_i b_0]$ for $0\leq i\leq p-1$. In particular, $U=\langle a_0,b_0\rangle$.
\item We see that $|U|=p^{(m-1)(p+1)+1}$.
\item The group $U$ is defined by the following relations in the generating set $\{a_0,b_0\}$:
\begin{gather*}
[[a_0,_i b_0],[a_0,_j b_0]]=1\text{ for }0\leq i<j\leq p-1,\\
[a_0,_p,b_0]=\prod_{j=1}^{p-1}[a_0,_j b_0]^{-\binom{p}{j}},\\
a_0^{p^m}=1,\\
[a_0,_i b_0]^{p^{m-1}}=1\text{ for }1\leq i\leq p-1,\\
b_0^{p^{m-1}}=1.
\end{gather*}
\item The group $U$ is solvable of class $2$. The subgroup $\langle b_0^p\rangle W$ is abelian and maximal.
\item The derived subgroup $U'$ of $U$ is given by
\begin{align*}
U'&=\langle a_1,a_2,\ldots a_{p-1}\rangle\\
&=\left\langle [a_0,_i b_0]:1\leq i\leq p-1\right\rangle.
\end{align*}
In particular, $U'$ is abelian of exponent $p^{m-1}$ and order $p^{(m-1)(p-1)}$.
\item The Frattini subgroup $\Phi(U)$ of $U$ is given by
\[
\Phi(U)=\langle a_0^p,a_1,a_2,\ldots a_{p-1},b_0^p\rangle.
\]
\item  Any non-trivial subgroup of $U'$ that is normal in $U$ must contain the {\rm(}non-trivial{\rm)} element $[a_0,_{p-1}b_0]^{p^{m-2}}$.
\end{enumerate}
\end{Lem}
\begin{proof}
Parts (i) to (vi) of the lemma are straighforward to prove, and we leave them to the reader. Note that Lemma~\ref{lem:basic_U_facts}~(iv) allows us to use Lemma~\ref{lem:abelian_index_p_facts} when reasoning about $U$.

To prove~(vii), note that the map $\kappa:U'\rightarrow U'$ defined by $\kappa(x)=[x,b_0]$ is a homomorphism with kernel $\langle a_{p-1}^{p^{m-2}}\rangle$ of order $p$.
Since $U$ is nilpotent, $\kappa^i$ is trivial for some positive integer $i$. Let $K\leq U'$ be a non-trivial subgroup of $U'$ that is normal in $U$. Let $i$ be the smallest integer such that $\kappa^i(K)=\{1\}$. Then $\kappa^{i-1}(K)=\langle a_{p-1}^{p^{m-2}}\rangle$. Since $K$ is normal in $U$, we have $\kappa^{i-1}(K)\leq K$ and so part~(vii) of the lemma follows.
\end{proof}

\begin{proof}[\bf{Proof of Theorem \ref{prevthm:primetypeA}}.]\label{thm:primetypeA}
Let $P$ be a non-abelian exponent critical $p$-group of exponent $p^m$ of type $\cal A$. We prove part~(i) of the theorem by showing that $P\cong U/N$ for some normal subgroup $N\in{\cal N}$.

By Theorem~\ref{thm:basic_pgroup} we have that $P=\langle a,b\rangle$ where $a,b\in P$ with $a$ of order $p^m$. Consider the free group $F$ generated by $a_0$ and $b_0$. Let $\pi$ be the (surjective) homomorphism from $F$ to $P$ mapping $a_0$ to $a$ and $b_0$ to $b$. We show that $P$ is a quotient $U/N$ of $U$ by checking that the relations defining $U$ are all satisfied under $\pi$.  (Here we are making use of the fact, see Lemma~\ref{lem:basic_U_facts}~(i), that $U=\langle a_0,b_0\rangle$.)

Let $A$ be the abelian maximal subgroup containing $a$. Since $P$ is not abelian, $b\in P\setminus A$. Lemma~\ref{lem:Type_B_basic_facts}~(i) shows that $b$ has order dividing $p^{m-1}$, since all elements of order $p^m$ lie in $A$. Hence $a^{p^m}=b^{p^{m-1}}=1$. Since $P'\leq \Phi(P)\leq A$, we see that the subgroup generated by $P'$ and $a$ is abelian. In particular, $[[a,_i b],[a,_jb]]=1$ for $0\leq i\leq j\leq p-1$. Lemma~\ref{lem:Type_B_basic_facts}~(iii) implies that $[a_0,_i,b_0]^{p^{m-1}}=1$ for $1\leq i\leq p-1$. Finally, Lemma~\ref{lem:abelian_index_p_facts}~(iii) shows that $[a_0,_p b_0]=\prod_{j=1}^{p-1}[a_0,_j b_0]^{-\binom{p}{j}}$. Hence $P\cong U/N$ for some normal subgroup $N$. Moreover, under this isomorphism the element $a\in P$ corresponds to the coset $a_0N$, and $b$ corresponds to $b_0N$.

We now check that the normal subgroup $N$ lies in $\cal N$. Since $a_0$ and $a$ both have order $p^m$, we see that $N\cap \langle a_0\rangle=\{1\}$. Since $P$ is non-abelian, we see that $U'\not\leq N$. Since $U$ and $P$ are $2$-generated, we see that $N\leq \Phi(U)$. So to prove part~(i) of the theorem, it remains to show that ${D}\leq N$. It suffices to show that $a^{p^{m-1}}[a,_{p-1}b]^{p^{m-2}}=1$ in $P$.

Consider the element $ba\in P$, since $ba\not\in A$, we see that $ba$ has order dividing $p^{m-1}$, by Lemma~\ref{lem:Type_B_basic_facts}~(i). Now, by Lemma~\ref{lem:abelian_index_p_facts}~(ii),
\begin{align*}
(ba)^p&=b^pa^p\prod_{j=1}^{p-1} [a,_j b]^{\binom{p}{j+1}}\\
&=a^p[a,_{p-1} b]\,b^p\left(\prod_{j=1}^{p-2}[a,_j b]^{\binom{p}{j+1}/p}\right)^p,
\end{align*}
since $p$ divides all the binomial coefficients in the product, and since all the factors lie in the abelian subgroup $A$. But $b$ has order dividing $p^{m-1}$ and the commutators all have order dividing $p^{m-1}$ by Lemma~\ref{lem:Type_B_basic_facts}~(iii). Hence $(ba)^p=a^p[a,_{p-1} b]x$ where $x\in A$ has order dividing $p^{m-2}$. So
\begin{align*}
1&=(ba)^{p^{m-1}}=\left((ba)^p\right)^{p^{m-2}}\\
&=a^{p^{m-1}}[a,_{p-1} b]^{p^{m-2}},
\end{align*}
as required. Hence part~(i) follows.

To prove part~(ii), we first investigate the orders of elements in the quotient $Q=U/{D}$. Since ${D}\cap\langle a_0\rangle=\{1\}$, the element $a_0D\in Q$ has order $p^m$. Hence the exponent of $Q$ is divisible by $p^m$. Define $M$ to be the subgroup $M:=\langle b_0^p\rangle W$. By Lemma~\ref{lem:basic_U_facts}~(iv), $M$ is maximal and abelian of exponent $p^{m}$. Since $D\leq M$, the subgroup $M/D$ is maximal in $Q$. Moreover, since $a_0D\in M/D$, we see that
$M/{D}$ of $Q$ has exponent $p^m$. Our next aim is to show that
\begin{equation}
\label{eqn:expt}
\text{Any element }g\in Q\setminus (M/{D})\text{ has order dividing }p^{m-1}.
\end{equation}

To see this, first write $g\in Q\setminus (M/{D})$ in the form $g=b_0^rx{D}$ where $1\leq r\leq p-1$ and $x\in M$. Replacing $g$ by $g^{r^{-1}\bmod p}$ does not change the order of $g$, and so we may assume without loss of generality that $r=1$. Since $M/{D}$ is an abelian maximal subgroup of $Q$, Lemma~\ref{lem:abelian_index_p_facts}~(ii) shows that
\[
(b_0x)^p{D}=b_0^px^p\prod_{j=1}^{p-1}[x,_j b_0]^{\binom{p}{j+1}}{D}.
\]
The factors in this product all lie in the abelian subgroup $M/{D}$ of $Q$. The factor $b_0^p{D}$ has order dividing $p^{m-2}$ as $b_0$ has order $p^{m-1}$. The factors $[x,_jb_0]^{\binom{p}{j+1}}{D}$ for $1\leq p-2$ have order dividing $p^{m-2}$, since (by Lemma~\ref{lem:basic_U_facts}~(v)) the derived subgroup of $U$ has exponent $p^{m-1}$ and $\binom{p}{j+1}$ is divisible by $p$. Hence
\[
(b_0x)^{p^{m-1}}{D}=x^{p^{m-1}}[x,_{p-1} b_0]^{p^{m-2}}{D}.
\]
Define
\[
S=\{x\in M: x^{p^{m-1}}[x,_{p-1} b_0]^{p^{m-2}}\in {D}\}.
\]
The map $x\mapsto x^{p^{m-1}}$ is a homomorphism on $M$. Moreover, the map $x\mapsto [x,b_0]$ is a homomorphism on $M$ by Lemma~\ref{lem:abelian_index_p_facts}~(i), and so the map $x\mapsto [x,_{p-1} b_0]$ is also a homomorphism. Thus the set $S$ is in fact a subgroup of $M$. To prove~\eqref{eqn:expt}, it suffices to show that $S=M$. Let $x=[a_0,_i b_0]$ where $1\leq i\leq p-1$. Then $x$ has order dividing $p^{m-1}$ since $x\in U'$. Now, the order of $[x,_{p-1}b_0]$ divides the order of $[x,_{p-i}b_0]$. But the order of $[x,_{p-i}b_0]$ divides $p^{m-2}$ since 
\begin{align*}
[x,_{p-i} b_0]&=[a_0,_{p}b_0]\\
&=\prod_{i=1}^{p-1} [a_0,_ib]^{-\binom{p}{i}}\\
&\in (U')^p.
\end{align*}
Hence $x=[a_0,_i b_0]\in S$ for $1\leq i\leq p-1$. By Lemma~\ref{lem:basic_U_facts}~(v), this implies that $U'\leq S$. Clearly $b_0^p\in S$, since the order of $b_0^p$ is $p^{m-2}$. Finally, $a_0\in S$ as $a_0^{p^{m-1}}[a_0,_{p-1}b_0]^{p^{m-2}}\in {D}$. Since $M$ is generated by $a_0$, $b_0^p$ and $U'$ we see that~\eqref{eqn:expt} holds, as required.

We are now in a position to prove part~(ii) of the theorem. Clearly $U/N$ is non-abelian, since $N$ does not contain $U'$. Since $U/{D}$ has exponent $p^m$, the exponent of $U/N$ divides $p^m$. But $U/N$ contains the element $a_0N$ of order $p^m$ (since $N\cap \langle a_0\rangle=\{1\}$), so $U/N$ has exponent $p^m$.

Since $N\leq \Phi(U)$, every maximal subgroup of $U/N$ is of the form $H/N$ for some maximal subgroup of $U$. Suppose a maximal subgroup $H/N$ contains an element $hN$ of maximal order $p^m$. Then $h\in U$ and since ${D} \leq N$, we get that $h{D}$ has order $p^m$.  By~\eqref{eqn:expt}, $h{D}\in M$. By Lemma~\ref{lem:basic_U_facts}~(vi), we may write $h=a_0^rx$ for some $x\in\Phi(U)$ where $1\leq r\leq p-1$. By replacing $h$ by a suitable power of $h$, we
may assume that $r=1$. Since $H$ is maximal, $H$ contains $\Phi(U)$, and so $\langle a_0,\Phi(U)\rangle \leq H$. But $\langle a_0,\Phi(U)\rangle=M$, a maximal subgroup of $U$, so $H=M$. But then $H/N$ is abelian, since it is a quotient of the abelian group $M$. This shows that $U/N$ is exponent-critical of exponent $p^m$, and so part~(ii) follows.

To prove the first statement of part~(iii), suppose that $p=2$ and $m=2$. For a contradiction, let $P$ be an exponent-critical $2$-group of exponent $2^2$ and of type $\cal A$. By part~(i), we have that $P\cong U/N$ for some $N\in{\cal N}$. Hence $|P|=|U/N|\leq |U/{D}|=2^3$. Since $P$ is a non-abelian group it must have order $8$ and is $2$-generated. So it cannot have a unique maximal subgroup. Further any maximal subgroups will have order $2^2$, and so are all abelian. Hence $P$ is not of type $\cal A$. This contradiction shows that the first statement of part~(iii) holds.

Assume now that $p$ is odd, or $m\geq 3$. Let $N\in {\cal N}$. Suppose, for a contradiction, that $N\cap U'$ is non-trivial. By Lemma~\ref{lem:basic_U_facts}~(vii), we see that $[a_0,_{p-1} b_0]^{p^{m-2}}\in N$. But $a_0^{p^{m-1}}[a_0,_{p-1} b_0]^{p^{m-2}}\in {D}\leq N$, and so $a_0^{p^{m-1}}\in N$. 
This contradicts the condition that $N\cap \langle a_0\rangle=\{1\}$. So $N\cap U'=\{1\}$ for all $N\in{\cal N}$. Hence $|(U/N)'|=|U'|=p^{(m-1)(p-1)}$.
Since we are assuming that $p$ is odd or $m\geq 3$, we see that $|(U/N)'|>p$. Since the derived subgroup of an exponent critical $p$-group of type $\cal B$ has order $p$, we see that $U/N$ has type ${\cal A}$ when $p$ is odd or $m\geq 3$. So part~(iii) of the theorem follows.
\end{proof}
\begin{Cor}
There is a unique largest non-abelian exponent-critical $p$-group of exponent $p^m$ of type ${\cal A}$. This $p$-group has cardinality $p^{(m-1)(p+1)}$.
\end{Cor}
\begin{proof}
    The corollary follows since ${D}\in {\cal N}$.
\end{proof}

\bibliographystyle{plain}

\end{document}